\documentclass[a4paper,11pt,reqno, english]{amsart}

\usepackage{anysize,color}
\usepackage[utf8]{inputenc}
\usepackage[dvipsnames]{xcolor}
\usepackage[colorlinks=true,urlcolor=blue,linkcolor={YellowOrange},citecolor={YellowOrange}]{hyperref}  
\usepackage{float}
\usepackage{amsfonts,amssymb,amsmath}
\usepackage{amsthm}    
\usepackage{url}
\usepackage{paralist}
\usepackage{tikz}
    \usetikzlibrary{decorations.markings,arrows,shapes.callouts}
\usepackage{rotating} 
\usepackage[mathscr]{euscript}
\usepackage{verbatim}
\usepackage{tikz-cd}
\usepackage{mathtools}
\usepackage{dsfont}
\usepackage{amsrefs}
\usepackage{thmtools}
\usepackage{thm-restate}

\usepackage{charter}

\newtheorem{theorem}{Theorem}[section]

\newtheorem{lemma}[theorem]{Lemma} 
\newtheorem{conjecture}[theorem]{Conjecture}

\theoremstyle{definition}

\theoremstyle{remark}

\newcommand{\N}{\mathcal{N}}
\newcommand{\M}{\mathcal{M}}
\newcommand{\I}{\mathcal{I}}

\def\keywords{\xdef\@thefnmark{}\@footnotetext}

\DeclareMathOperator{\Circ}{\mathrm{circ}}

\tikzcdset{row sep/my_size/.initial=-2mm}
\tikzcdset{row sep/my_size_small/.initial=-3mm}
\newcommand{\xDownarrow}[1]{%
  {\left\Downarrow\vbox  to #1{}\right.\kern-\nulldelimiterspace}
}

\title{A common generalization to strengthenings of Drisko's Theorem for intersections of two matroids}
\author{Eli Berger and Daniel McGinnis}
\date{}

\begin{document}

\maketitle

\begin{abstract}
        Let $\M$ and $\N$ be two matroids on the same ground set $V$. Let $A_1,\dots,A_{2n-1}$ be sets which are independent in both $\M$ and $\N$, satisfying $|A_i|\geq \min(i,n)$ for all $i$. We show that there exists a partial rainbow set of size $n$, which is independent in both $\M$ and  $\N$. This is a common generalization of rainbow matching results for bipartite graphs by Aharoni, Berger, Kotlar, and Ziv, and for the intersection of two matroid by Kotlar and Ziv. 
\end{abstract}

\section{Introduction}

Let $(M_1,\dots,M_n)$ be a collection (possibly with repetition) of matchings (vertex disjoint edges) in a bipartite graph $G$. A \textit{partial rainbow matching} of size $m$ is a set of pairwise vertex disjoint edges $e_1,\dots,e_m$ such that $e_1\in M_{i_1},\dots, e_m\in M_{i_m}$ where $1\leq i_1 < \cdots<i_m\leq n$.
The following conjecture of Aharoni and Berger from \cite{aharoni2009rainbow} on rainbow matchings is of significant interest (see also Conjecture 1.6 in \cite{aharoni2018degree}).

\begin{conjecture}[Berger-Aharoni \cite{aharoni2009rainbow}]\label{conj:n matchings}
    A collection of $n$ matchings of size $n$ in a bipartite graph has a partial rainbow matching of size $n-1$.
\end{conjecture}
The fact that the number $n-1$ in Conjecture \ref{conj:n matchings} cannot be replaced by $n$ follows from the following example. Let $n$ be even and let $G$ be the complete bipartite graph with both partite sets equal to $[n]=\{1,2, 3, \dots, n\}$. Note that we think of the edges in such a graph as ordered pairs, where $(i,j)$ means we take the vertex $i$ from the first partite set and the vertex $j$ from the second partite set. The collection of $n$ matchings $M_i$ for $1\leq i\leq n$, where $M_i$ is given by the set of edges $(j,j+i)$ (taking addition modulo $n$), does not have a rainbow matching of size $n$. 

Conjecture \ref{conj:n matchings} is closely related to well-known open problems concerning transversals of latin squares by Ryser, Brualdi, and Stein \cite{Ryser1967,brualdi1991combinatorial,stein1975transversals}. See for instance \cite{aharoni2009rainbow} for the statements of these conjectures and how they are related to Conjecture \ref{conj:n matchings}.

One may naturally ask how many matchings of size $n$ of a bipartite graph $G$ are needed to guarantee a partial rainbow matching of size $n$. Interestingly, almost twice as many matchings are required.

\begin{theorem}[Aharoni-Berger \cite{aharoni2009rainbow}]\label{thm:bipartitie Drisko}
    A collection of $2n-1$ matchings of size $n$ in a bipartite graph has a partial rainbow matching of size $n$.
\end{theorem}

The following example shows that the number $2n-1$ in Theorem \ref{thm:bipartitie Drisko} cannot be replaced by a smaller number. Let $G$ be the cycle on $2n$ vertices and take $n-1$ copies of one maximum matching and $n-1$ copies of the other maximum matching. There is no rainbow partial matching of size $n$ for this collection of matchings.

Theorem \ref{thm:bipartitie Drisko} is a generalization of the following well-known theorem of Drisko. For the purpose of stating this result, we say that a \textit{diagonal} of a matrix is a set of entries with at most one entry in each column and at most one entry in each row. 

\begin{theorem}[Drisko \cite{drisko1998transversals}]\label{thm: drisko}
    Let $X$ be an $n\times (2n-1)$ matrix such that each column contains each of the integers $1,\dots,n$. Then there is a diagonal of $X$ that contains each of the integers $1,\dots,n$.
\end{theorem}

To see that Theorem \ref{thm:bipartitie Drisko} implies Theorem \ref{thm: drisko}, note that each column of the matrix $X$ corresponds to a matching of the complete bipartite graph $G$ with both partite sets equal to $[n]$, where $(i,j)$ is an edge of the matching if $j$ is the entry in the $i$'th row of the column. A rainbow matching of size $n$ corresponds to a diagonal as in the conclusion of Theorem \ref{thm: drisko}.

A matroid-theoretic generalization of Theorem \ref{thm: drisko} was proven soon after by Chappel (see Section \ref{preliminaries} for background on matroids).

\begin{theorem}[Chappel \cite{chappell1999matroid}]\label{thm:chappel}
    Let $X$ be an $n\times (2n-1)$ matrix whose entries are elements of the ground set $V$ of a matroid $\M$. Then there is a diagonal of $X$ of size $n$ that is independent in $\M$
\end{theorem}

Theorem \ref{thm:chappel} is easily seen to be a generalization of Theorem \ref{thm: drisko} by taking $\M$ to be the matroid with ground set $[n]$ where each subset is independent.

Kotlar and Ziv \cite{kotlar2015rainbow} took this one step further by proving the following theorem. Here, if $S_1,\dots,S_n$ are sets, a \textit{partial rainbow set} is a selection of elements $x_1\in S_{i_1},\dots,x_m\in S_{i_m}$ where $1\leq i_1<\cdots<i_m\leq n$.

\begin{theorem}[Kotlar-Ziv \cite{kotlar2015rainbow}]\label{thm:2matroids}
    Let $\M$ and $\N$ be matroids on the same ground set $V$. For any $2n-1$ subsets of $V$ of size $n$ that are independent in both $\M$ and $\N$, there is a partial rainbow set of size $n$ that is independent in both $\M$ and $\N$.
\end{theorem}

Theorem \ref{thm:chappel} follows from Theorem \ref{thm:2matroids} by the following argument. Let $V$ be the ground set of a matroid $\M$. Let $\M'$ be the matroid on the ground set $V\times [n]$ whose independent sets are those sets with elements that have distinct values in their first coordinate and whose image under projection onto the first coordinate is independent in $\M$. Let $\N$ be the matroid on $V\times [n]$ whose independent sets are those sets with elements that have distinct values in their second coordinate. Given a matrix $X$ as in Theorem \ref{thm:chappel}, define the set $A_i$ for $1\leq i\leq 2n-1$ to be the set of elements $(v,j)$ where $v=X_{ji}$. Then $A_i$ is independent in $\M'$ and $\N$ for all $i$ and thus has a partial rainbow set of size $n$ that is independent in both $\M'$ and $\N$. This corresponds to the diagonal as in the conclusion of Theorem \ref{thm:chappel}.

Theorem \ref{thm:bipartitie Drisko} also follows from Theorem \ref{thm:2matroids} by the following observation. Let $G$ be a bipartite graph with partite sets $A$ and $B$. Let $V$ be the edge set of $G$ and let $\M$ be the matroid with ground set $V$ whose independent sets are the sets of edges whose endpoints in $A$ are all distinct. Define the matroid $\N$ similarly where instead the endpoints in $B$ are considered. Then a matching of $G$ is precisely a set of edges that is independent in both $\M$ and $\N$.

Theorem \ref{thm:bipartitie Drisko} was further generalized by Aharoni, Berger, Kotlar, and Ziv in \cite{aharoni2018degree} by showing that the size of some the matchings can be reduced while the same conclusion still holds.

\begin{theorem}[Aharoni et. al. \cite{aharoni2018degree}]\label{thm:strong bip drisko}
    Let $(M_1,\dots,M_{2n-1})$ be matchings in a bipartite graph $G$ satisfying $|M_i| \geq \min(i,n)$ for all $i$. Then there is a partial rainbow matching of size $n$.
\end{theorem}

Additionally, it was shown in \cite{aharoni2018degree} that Theorem \ref{thm:strong bip drisko} is tight in the following sense. If the matchings $M_1,\dots,M_{2n-1}$ are ordered in nondecreasing order in terms of their size, then the conclusion of Theorem \ref{thm:strong bip drisko} does not hold in general if $|M_i| < \min(i,n)$ for some $i$.

The main theorem of this paper is a common generalization of Theorems \ref{thm:2matroids} and \ref{thm:strong bip drisko}.

\begin{theorem}[Main theorem]\label{thm:main}
    Let $\M$ and $\N$ be two matroids on the same ground set $V$. Let $A_1,\dots,A_{2n-1}$ be sets which are independent in both $\M$ and $\N$, satisfying $|A_i|\geq \min(i,n)$ for all $i$. Then there exists a partial rainbow set of size $n$, which is independent in both $\M$ and $\N$.
\end{theorem}

\section{Preliminaries}\label{preliminaries}

In this section, we introduce the necessary background and supporting results needed to prove Theorem \ref{thm:main}. A hypergraph $H=(V,E)$ consists of a \textit{ground set} $V$ (for our purposes $V$ will be finite) and an \textit{edge set} $E$ consisting of subsets of $V$. A \textit{matroid} is a hypergraph whose edge set $E$ satisfies the following properties:
\begin{enumerate}
    \item $\emptyset \in E$,
    \item If $T\in E$ and $S\subseteq T$, then $S\in E$,
    \item If $S,T\in E$ and $|S| < |T|$, then there exists $x\in T\setminus S$ such that $S\cup \{x\} \in E$. 
\end{enumerate}
The \textit{rank} in $\M$ of a set $S\subseteq V$ is the size of the largest edge of $\M$ contained in $S$, and we denote it as $\rho_\M(S)$. The rank of $\M$, denoted $\rho(\M)$, is defined to be $\rho_\M(V)$. A \textit{basis} of $\M$ is an edge $e$ of $\M$ such that $|e|$ is equal to the rank of $\M$. A \textit{flat} $F$ of $\M$ is a subset of $V$ satisfying $\rho_\M(F\cup\{x\}) > \rho_\M(F)$ for all $x\in V\setminus F$. A \textit{loop} is an element $x\in V$ such that $\{x\}\notin E$.  The set of \textit{circuits} of $\M$, which we denote $\Circ(\M)$, are the minimal sets (with respect to inclusion) that are not an edge of $\M$. A \textit{coloop} is an element that does not belong to any circuit. 

Note that we may also view $\Circ(\M)$ as a hypergraph with ground set $V$. We may abuse notations by writing $\Circ(\M)$ instead of $(V,\Circ(\M))$. A well-known property of matroids is that they satisfy the \textit{circuit elimination axiom}.
\begin{itemize}
    \item(circuit elimination axiom) For any two circuits $C_1$ and $C_2$ with $x\in C_1\cap C_2$ and $y\in C_1\setminus C_2$, there exists a circuit $C_3$ such that $C_3\subseteq C_1\cup C_2$ and $x\notin C_3$ and $y\in C_3$.
\end{itemize}
See \cite{oxley2011} for a more detailed account of matroids.

For a general hypergraph $H=(V,E)$, we call a set $S\subseteq V$ independent if there is no edge $e\in E$ such that $e\subseteq S$. We denote by $\I(H)$ the hypergraph on the same ground set whose edges are the independent sets of $H$. Note that for a matroid $\M$ we have $\M=\I(\Circ(\M))$. For an edge $e\in E$, we define 
\[
H - e= (V,E\setminus\{e\}).
\] 
For a set $S\subseteq V$, we define
\[
H/S = (V\setminus S, \{e\setminus S \mid e\in E,\, e \not\subseteq S\})
\]
For two hypergraphs $H_1=(V,E_1), H_2=(V,E_2)$ on the same ground set $V$, we denote $H_1\cap H_2$ as the hypergraph $(V,E_1\cap E_2)$. For a matroid $\M$, we define $\M/S = \I(\Circ(\M)/S)$. Note that $\M/S$ coincides with the usual notion of contraction in matroids, and hence is itself a matroid.

We will also require some topological preliminaries. Homology groups will be taken with $\mathbb{Q}$ coefficients. A \textit{simplicial complex} is a collection of finite sets that is closed under taking subsets, and a set of the simplicial complex may be called a simplex. For a simplicial complex $\mathcal{C}$, $\eta(\mathcal{C})$ is the smallest integer $k$ such that the reduced homology $\tilde{H}_{k-1}(\mathcal{C})$ is not zero. If $\mathcal{C}$ is empty we take $\eta(\mathcal{C})=0$ and if all homology groups of $\mathcal{C}$ vanish, then we take $\eta(\mathcal{C})=\infty$.

The following result will allow us to obtain lower bounds on $\eta$ and is key to the proof of the main theorem.

\begin{theorem}\label{thm:eta bound}
    Let $H$ be a hypergraph, and let $e$ be an edge that does not contain another edge. Then
    \[
    \eta(\I(H)) \geq \min\Big(\eta\big(\I(H-e)\big),\eta\big(I(H/e)\big)+|e|-1\Big).
    \]
\end{theorem}
Theorem \ref{thm: drisko} was proven by Meshulam in \cite{meshulam2003domination} for graphs, but the proof holds for general hypergraphs as well. It was used implicitly in \cite{aharoni2008acyclic} and explicitly in \cite{berger2024coloring}.

Another result that is necessary for our proof is the following theorem of Aharoni and Berger that provides a necessary condition for the existence of a set that is both a basis of a matroid and an element of a simplicial complex both with the same ground set $V$. For a matroid $\M$ and a set $S\subseteq V$, we write $\M.S$ for the matroid consisting of the sets $e\subseteq S$ such that $e\cup f$ is an edge of $\M$ for all edges $f$ with $f\cap S= \emptyset$. For a simplicial complex $\mathcal{C}$ and a set $S$ contained in its ground set, we write $\mathcal{C} | S$ for the simplicial complex given by the collection of sets of $\mathcal{C}$ contained in $S$.

\begin{theorem}[Aharoni-Berger \cite{aharoni2006intersection}]\label{thm:matchability}
    Let $\M$ be a matroid and $\mathcal{C}$ be a simplicial complex with the same ground set $V$. If $\eta(\mathcal{C} | S) \geq \rho(\M.S)$ for all $S\subseteq V$, then there is a basis of $\M$ in the edge set of $\M\cap \mathcal{C}$. In fact, it suffices to assume the condition holds for sets $S$ that are complements of a flat of $\M$.
\end{theorem}

\section{Proof of Theorem \ref{thm:main}}
We are now ready to prove the main theorem. We will need Lemma \ref{lem:main} below to provide lower bounds on $\eta$, which will allow us to later apply Theorem \ref{thm:matchability}. Given sets $X_1\dot\cup\cdots \dot\cup X_m$ ($\dot\cup$ denotes disjoint union), the \textit{partition matroid} $P$ is the matroid whose edges are all partial rainbow sets of the $X_i$. 
\begin{lemma}\label{lem:main}
    Let $\ell\geq 1$ and $m \geq 2\ell-1$ be integers. Let $P$ be the partition matroid on $X_1\dot\cup \cdots \dot\cup X_{m}$. Let $\N$ be a matroid on $V \supseteq X_1\cup \cdots \cup X_{m}$, and assume that there are indices $i_1,\dots,i_{2\ell-1}$ such that $\rho_\N(X_{i_j})\geq \min(j,\ell)$. Then $\eta(P\cap \N)\geq \ell$.
\end{lemma}
\begin{proof}
We proceed by induction on $\ell$; for the base case $\ell = 1$ we only need to show that there exists some non-empty set in $P\cap \N$, and this follows from the condition $\rho_\N(X_{i_1})\geq 1$.   

We now need to show the induction step. 
For notational purposes, we will assume without loss of generality that $i_j=j$. Furthermore,  by replacing $\N$ with $\N|(X_1\cup \cdots \cup X_m)$ if necessary, we can assume that $X_1\cup\cdots\cup X_m = V$.

Let $H$ be the hypergraph on the ground set $V$ whose edges are the circuits of $\N$ and the 2-element subsets $\{x,y\}$ where $x,y\in X_i$ for some $i$. Note that $P\cap \N = \I(H)$.

Let $x\in X_1$ be a non-loop. The existence of such a non-loop follows from the assumption $\rho_\N(X_1)\geq 1$. Let $C_1,\dots,C_r$ be the circuits of $\N$ that contain $x$ in any order. We apply Theorem \ref{thm:eta bound} iteratively to $H$ with the edges $C_1,\dots,C_r$ one at a time. Let $H_0=H$, and for $i\geq 1$, let $H_i=H_{i-1}-C_i$ or $H_i=H_{i-1}/C_i$ depending on which one provides a lower bound for $\eta(\I(H_{i-1}))$ as in Theorem \ref{thm:eta bound}. If $H_i=H_{i-1}/C_i$ for some $i$ then we stop the process. Assume first that there exists some $i$ such that $H_i= H_{i-1}/C_i$. We claim that $\I(H_i)=\I(H/C_i)$. It is clear that $\I(H_i)\supseteq \I(H/C_i)$ since the edge set of $H_i$ is contained in the edge set of $H/C_i$. If there exists an edge $e \in \I(H_i)\setminus \I(H/C_i)$, then $e$ contains a subset of the form $C_j\setminus C_i$ for some $j<i$. However, by the circuit elimination property, there exists a circuit $C'$ such that $x\notin C'$, $C'\subseteq C_j\cup C_i$, and $C'\not\subseteq C_i$. Hence, $C'\setminus C_i\subseteq C_j\setminus C_i$, contradicting the fact that $e \in \I(H_i)$. Therefore, $\eta(\I(H))\geq \eta(\I(H_i)) + |C_i| -1 = \eta(\I(H/C_i)) + |C_i| -1$. Let $k=\rho_\N(C_i)$. If $k\geq \ell$, then $|C_i|\geq \ell +1$ and we are done. Otherwise, for the indices $j$ for which $X_j\cap C_i=\emptyset$ we have that $\rho_{\N/C_i}(X_j)\geq j-k$ for all $j$. In particular, there are $2(\ell-k)-1$ indices $i_1,\dots,i_{2(\ell-k)-1}$ such that $\rho_{\N/C_i}(X_{i_j})\geq j-k$ for all $j$. Thus, by induction applied to $\N/C_i$ and the $X_{j}$ with $X_j\cap C_i=\emptyset$, we have $\eta(\I(H/ C_i))\geq \ell-k$, which implies that $\eta(\I(H))\geq \ell$.

Now, assume instead that $H_i = H_{i-1}-C_i$ for all $i$, and let $H'=H_r$. Let $e_1,\dots,e_s$ be the edges of the form $\{x,y\}$ where $y\neq x$ and $y\in X_1$ in any order. Define $H_i'$ for $0\leq i\leq s$ in a similar way as before. If $H_i' = H_{i-1}'-e_i$ for all $i$, or if no such edge exists, then $x$ is not contained in any edge of $H_s'$. This implies that $\I(H_s')$ is a cone and is contractable as a simplicial complex, which means $\eta(\I(H)) \geq \eta(\I(H_s')) = \infty$. Thus, we may assume that $H_i'=H_{i-1}'/e_i$ for some $i$, and let $e_i=\{x,y\}$. Similarly to before, we have $\I(H_i')=\I(H'/e_i)$. Let $\N'$ be the matroid obtained by removing all circuits of $\N$ containing $x$. Since $x$ is a coloop in $\N'$, we have $\rho_{\N'/\{x,y\}}(X_j) \geq j-1$ for all $j$. Thus, by induction applied to $\N'/\{x,y\}$ and $X_2,\dots,X_m$, $\eta(\I(H'/e_i))\geq \ell -1$. Therefore, $\eta(\I(H))\geq \eta(\I(H')) \geq \eta(\I(H'/e_i)) +1 \geq \ell$.

\end{proof}

\begin{proof}[Proof of Theorem \ref{thm:main}]
We will assume that $\M$ is a matroid of rank $n$, which we can do without loss of generality since we can always replace $\M$ with the matroid whose independent sets consist of the independent sets of $\M$ with size at most $n$.

Let $A_1,\dots,A_{2n-1}$ be independent sets of $\M$ and $\N$ as in the statement of the theorem. Let $\mathcal{C}$ be the simplicial complex with vertex set $V'$ given by the set of pairs $\{(x,i)\mid x\in A_i\}$ and with simplices of the form $\{(x_1,i_1),\dots,(x_m,i_m)\}$ where the $i_j$'s, as well as the $x_j$'s, are pairwise distinct, and $\{x_1,\dots,x_m\}$ is independent in $\N$. Let $\M'$ be the matroid on the same ground set $V'$ whose independent sets are sets of the form $\{(x_1,i_1),\dots,(x_m,i_m)\}$ where the $x_j$'s are pairwise distinct and $\{x_1,\dots,x_m\}$ is independent in $\M$ (the $i_j$'s do not have to be distinct here). Define $\N'$ similarly.

Notice that the existence of a set that is simultaneously a basis of $\M'$ and a simplex of $\mathcal{C}$ gives us a partial rainbow set as in the conclusion of Theorem \ref{thm:main} and hence, would complete the proof. We will apply Theorem \ref{thm:matchability} to prove the existence of such a set. 

Let $F'$ be a flat of $\M'$ of rank $k$ and let $S'$ be its complement. Notice that there is a corresponding flat $F$ of $\M$ such that $F'=\{(x,i)\mid x\in F,\, 1\leq i\leq 2n-1\}$. It is clear from the definition of $\M'.S'$ that $\rho(\M'.S') \leq n-k$, so it suffices to prove that $\eta(\mathcal{C}|S') \geq n-k$. Let $S$ be the complement of $F$. We have  $\rho_{\M}(A_i \cap S)\geq \rho_\M(A_i) - k$, which implies that $|A_i\cap S|\geq |A_i|-k$. Therefore, the lower bound on $\eta(\mathcal{C}|S)$ follows from applying Lemma \ref{lem:main}  to the matroid $\N'|S'$ and the sets $X_i=\{(x,i)\mid x\in A_i\cap S\}$ for $1\leq i\leq 2n-1$. This completes the proof.

\end{proof}


\end{document}